\documentclass[12pt]{article}
\usepackage{amsmath,amssymb,amsfonts}
\usepackage{epsfig}
\usepackage{bbm}

\newtheorem{theorem}{Theorem}[section]
\newtheorem{lemma}[theorem]{Lemma}

\newtheorem{proposition}[theorem]{Proposition}

\newtheorem{corollary}[theorem]{Corollary}

\newenvironment{proof}{\noindent\emph{Proof.}\hspace{.25em}}{\hspace*{\fill}
$\Box$\newline}

\def\pfssp{\hskip 0.5em}
\newcommand{\pff}[2]{\noindent{\it Proof~of\,\,#1\,:} \pfssp #2 \qed \smallskip}
\def \qed {\hfill $\Box$}
\def \QD1 {\hfill $\spadesuit$}

\newcommand{\case}[2]{\smallskip {\bf Case #1\/:} {\it #2}}

\newcommand{\set}[2]{\{#1 \;|\; #2 \}}
\newcommand{\ems}{\varnothing}
\newcommand{\sm}{\setminus}

\newcommand{\De}{\Delta}
\newcommand{\de}{\delta}

\newcommand{\ep}{\varepsilon}

\newcommand{\sg}{\sigma}
\newcommand{\scn}{\chi^{\pm}}
\newcommand{\sln}{\chi_{\ell}^{\pm}}

\newcommand{\ul}[1]{\underline{#1}}

\newcommand{\pa}{\partial}
\newcommand{\f}{\phi}
\newcommand{\Ga}{Z}

\newcommand{\cB}{{\cal B}}
\newcommand{\cT}{{\cal T}}

\newcommand{\col}{{\rm col}}

\newcommand{\nato}{\mathbb{N}_0}

\newcommand{\mz}{\mathbb{Z}}
\newcommand{\pl}{\mathbbm{1}}

\numberwithin{equation}{section}

\begin{document}
\title{\bf Degree choosable signed graphs }

\author{{{
Thomas Schweser}\thanks{
Technische Universit\"at Ilmenau, Inst. of Math., PF 100565, D-98684 Ilmenau, Germany. E-mail
address: thomas.schweser@tu-ilmenau.de}}
\and{{Michael Stiebitz}
\thanks{
Technische Universit\"at Ilmenau, Inst. of Math., PF 100565, D-98684 Ilmenau, Germany. E-mail
address: michael.stiebitz@tu-ilmenau.de}}
}
\date{}
\maketitle

\begin{abstract}
A signed graph is a graph in which each edge is labeled with $+1$ or $-1$. A (proper) vertex coloring of a signed graph is a mapping $\f$ that assigns to each vertex $v\in V(G)$ a color $\f(v)\in \mz$ such that every edge $vw$ of $G$ satisfies $\f(v)\not= \sg(vw)\f(w)$, where $\sg(vw)$ is the sign of the edge $vw$. For an integer $h\geq 0$, let $\Ga_{2h}=\{\pm1,\pm2, \ldots, \pm h\}$ and $\Ga_{2h+1}=\Ga_{2h} \cup \{0\}$. Following \cite{MaRS2015}, the signed chromatic number $\scn(G)$ of $G$ is the least integer $k$ such that $G$ admits a vertex coloring $\f$ with ${\rm im}(\f)\subseteq \Ga_k$. As proved in \cite{MaRS2015}, every signed graph $G$ satisfies $\scn(G)\leq \De(G)+1$ and there are three types of signed connected simple graphs for which equality holds. We will extend this Brooks' type result by considering graphs having multiple edges. We will also proof a list version of this result by characterizing degree choosable signed graphs. Furthermore, we will establish some basic facts about color critical signed graphs.

\end{abstract}



\section{Introduction}
This paper deals with the vertex coloring problem for signed graphs introduced by Zaslavsky \cite{Zaslavsky82a,Zaslavsky82b,Zaslavsky84} in the 1980s. Recently M\'a\v{c}ajov\'a, Raspaud and \v{S}koviera \cite{MaRS2015}
proved an extension of Brooks' theorem to signed simple graphs. Our aim is to characterize signed graphs that are degree choosable and to establish some basic properties of critical signed graphs.

\bigskip

\centerline{\bf Signed graphs}

\medskip

Signed graphs were first defined and investigated by Harary \cite{Harary53}.
Throughout this paper, the term graph refers to a finite graph which may have multiple edges but no loops. A {\em signed graph} is a graph in which the edges are labeled by $+1$ or $-1$. So a singed graph is a triple $G=(V,E,\sg)$, where $V=V(G)$ is the vertex set of $G$, $E=E(G)$ is the edge set of $G$, and $\sg=\sg_G$ is the sign mapping of $G$, i.e. $\sg:V(G) \to \{-1,1\}$. In order to make a clear distinction between a signed graph and its underlying graph, we shall use $\ul{G}$ to denote the underlying graph of a signed graph $G$.

For a signed graph $G$ we adopt the standard notations for graphs.
For $X, Y \subseteq V(G)$, let $E_G(X,Y)$ be the set of all edges joining a vertex of $X$ with a vertex of $Y$. Let $E_G[X]=E_G(X,X)$ be the set of all edges with both ends in $X$, and let $\pa_GX$ the set of all edges with exactly one end in $X$.
If the meaning is clear we will frequently omit subscripts and brackets for the sake of readability. Thus, the degree of a vertex $v$ in $G$ is $d_G(v)=|\pa_G v|$, and the multiplicity of two distinct vertices $v, w$ in $G$ is $\mu_G(v,w)=|E_G(v,w)|$. As usual, we denote by $\De(G)$, $\de(G)$ and $\mu(G)$ the maximum degree, the minimum degree and the maximum multiplicity of $G$, respectively. Thus, a simple signed graph is a signed graph $G$ with $\mu(G)\leq 1$.

\bigskip

\centerline{\bf Switching and balance}

\medskip

An edge of a signed graph $G$ is said to be {\em positive} or {\em negative} depending on whether its sign is +1 or -1. A {\em positive signed graph} is a signed graph consisting only of positive edges; and a {\em negative signed graph} is a signed graph consisting only of negative edges. A signed graph is positive if and only if $\sg=\pl$, i.e., $\sg(e)=1$ for all $e\in E(G)$.

If $X$ is a vertex set of a signed graph $G$, then a new signed graph $G'$ can be obtained by reversing the sign of each edge belonging to the coboundary $\pa_G X$. Then $\ul{G'}=\ul{G}$, $\sg'(e)=-\sg(e)$ for $e\in \pa_GX$ and $\sg(e)=\sg'(e)$ otherwise. We then say that $G'$ is obtained from $G$ by {\em switching} at $X$ and write $G'=G/X$. Let $X$ and $Y$ be subsets of $V(G)$. Furthermore, let $\overline{X}=V(G)\sm X$ and $X + Y=(X\cup Y)\sm (X \cap Y)$. Then $G/X=G/\overline{X}$, $G/\ems=G$ and $G/X/Y=G/(X + Y)$. Two signed graphs $G$ and $G'$ are {\em switching equivalent}, written $G\equiv G'$, if there is a vertex set $X\subseteq V(G)$ such that $G'=G/X$. This obviously defines an equivalence relation for the class of signed graphs.

If $G$ is a signed graph and $H$ is a signed subgraph of $G$, that is, $\ul{H}$ is a subgraph of $\ul{G}$ and $\sg_H=\sg_G|_{E(H)}$, then $$\sg_G(H)=\prod_{e\in E(H)}\sg_G(e)$$
is called the {\em sign product} of $H$. A signed graph $G$ is called {\em balanced} if the sign product of each cycle of $G$ is positive, otherwise it is called {\em unbalanced}. Clearly, if a signed graph $G$ is balanced, then any signed graph switching equivalent to $G$ is balanced, too. The following characterization of balanced graphs was obtained by Harary \cite{Harary53}.

\begin{theorem}
\label{Theorem:Harary:balanced}
For a signed graph the following statements are equivalent:
\begin{itemize}
\item[{\rm (a)}] $G$ is balanced.
\item[{\rm (b)}] The vertex set of $G$ is the disjoint union of two sets $X$ and $Y$ such that an edge of $G$ is negative if and only if this edge belongs to $E_G(X,Y)$.
\item[{\rm (c)}] $G$ is switching equivalent to a positive signed graph.
\end{itemize}
\end{theorem}

If a signed graph $G$ satisfies statement (b) of the above theorem, we say that $G$ is a {\em balanced graph with parts} $X$ and $Y$.

A signed graph $G$ is {\em antibalanced} if the sign product of every even cycle of $G$ is positive and the sign product of every odd cycle of $G$ is negative.
The {\em negation} of a signed graph $G$ is the signed graph obtained from $G$ by reversing the sign of all edges of $G$. Obviously, a signed graph is antibalanced if and only if its negation is balanced. So Harrary's characterization of balanced graphs implies the following result.

\begin{theorem}
\label{Theorem:Harary:antibalanced}
For a signed graph the following statements are equivalent:
\begin{itemize}
\item[{\rm (a)}] $G$ is antibalanced.
\item[{\rm (b)}] The vertex set of $G$ is the disjoint union of two sets $X$ and $Y$ such that an edge of $G$ is positive if and only if this edge belongs to $E_G(X,Y)$.
\item[{\rm (c)}] $G$ is switching equivalent to a negative signed graph.
\end{itemize}
\end{theorem}

\centerline{\bf Signed chromatic number}

\medskip

Let $G$ be a signed graph. A {\em coloring} of $G$ is a mapping $\f:V(G) \to \mz$ such that every edge $e\in E_G(v,w)$ satisfies $\f(v)\not=\sg(e) \f(w)$. It is notable that if two vertices $v$ and $w$ of $G$ are joined by a pair of differently signed parallel edges, then $|\f(v)|\not=|\f(w)|$. The above definition is due to Zalavsky \cite{Zaslavsky82a} and was mainly motivated by the following two simple observations. A coloring of a positive signed graph is an ordinary vertex coloring of its underlying graph. Furthermore, if $\f$ is a coloring of a signed graph $G$ and $G'=G/X$ for a vertex set $X$ of $G$, then the mapping $\f'$, satisfying $\f'(v)=-\f(v)$ if $v\in X$ and $\f'(v)=\f(v)$ otherwise, is a coloring of $G'$. We denote the coloring $\f'$ by $\f/X$.

A subset $C$ of $\mz$ is also called a {\em color set}. Then $-C=\set{-c}{c\in C}$ and the color set $C$ is called {\em symmetric} if $C=-C$. For an integer $h\geq 0$, let $\Ga_{2h}=\{\pm1,\pm2, \ldots, \pm h\}$ and $\Ga_{2h+1}=\Ga_{2h} \cup \{0\}$. The {\em signed chromatic number} $\scn(G)$ of $G$ is the least integer $k$ such that $G$ admits a coloring $\f$ with ${\rm im}(\f)\subseteq \Ga_k$.

The above  definition of the signed chromatic number is due to M\'a\v{c}ajov\'a, Raspaud and \v{S}koviera \cite{MaRS2015}. They also established some basic facts about the signed chromatic number. By the two observations about colorings of signed graphs it follows that switching equivalent signed graphs have the same signed chromatic number and the signed chromatic number
of a balanced signed graph coincides with the chromatic number $\chi$ of its underlying graph. As proved in \cite{MaRS2015}, if $G$ is a signed graph, then
\begin{equation}
\label{Equation:chi+<chi}
\scn(G)\leq 2\chi(\ul{G})-1
\end{equation}
and there are signed simple graphs for which equality hold. We want to introduce a class
of signed graphs for which equality holds in (\ref{Equation:chi+<chi}). If $H$ is a simple graph, then we denote by $G=2H$ the signed graph obtained from $H$ by replacing every edge of $H$ by a pair of differently signed parallel edges. Then we have the following result.

\begin{theorem}
\label{Theorem:G=2H}
If $G=2H$ for a simple graph $H$, then $\scn(G)=2\chi(H)-1$.
\end{theorem}
\begin{proof}
Let $k=\chi(H)$ and $h=\scn(G)$. Then there is an ordinary vertex coloring $\f$ of $H$ using colors $0,1, \ldots, k-1$. Obviously, $\f$ is a coloring of the signed graph $G$ with ${\rm im}(\f) \subseteq \Ga_{2k-1}$. Hence $\scn(G)\leq 2\chi(H)-1$. Since $h=\scn(G)$, there exists a coloring $\f$ of $G$ with ${\rm im}(\f)\subseteq \Ga_h$. If $X=\set{v\in V(G)}{\f(v)<0}$, then $G/X=G$ and so $\f'=\f/X$ is a coloring of $G$ with ${\rm im}(\f')\subseteq \Ga_h$ and $\f'(v)\geq 0$ for all $v\in V(G)$. Then $\f'$ is an ordinary vertex coloring of $G$ and we obtain that $k\leq \frac{h+1}{2}$, i.e., $\scn(G)\geq 2\chi(G)-1$.
\end{proof}

As an immediate consequence of Theorem~\ref{Theorem:Harary:antibalanced} we obtain the following characterization of signed graphs with $\scn\leq 2$.

\begin{theorem}
\label{Theorem:Bipartite}
A signed graph $G$ satisfies $\scn(G)\leq 2$ if and only if $G$ is antibalanced.
\end{theorem}

The {\em coloring number} $\col(G)$ of a (signed) graph $G$ is the maximum minimum degree of the subgraphs of $G$ plus 1. A graph with coloring number at most $k+1$ is also called {\em $k$-degenerate}. It is well known and easy to prove that a graph is $k$-degenerate if and only if there is an ordering $v_1,v_2, \ldots , v_n$ of its vertices such that for every $i\in \{1,2, \ldots, n\}$ the vertex $v_i$ has degree at most $k$ in the subgraph induced by the vertex set $\{v_1,v_2, \ldots, v_i\}$. Using the classical sequential coloring procedure, it is easy to see that every signed graph $G$ satisfies
$$\scn(G)\leq \col(G).$$
That this bound holds was observed in \cite{MaRS2015}. As a consequence, every signed graph $G$ satisfies $\scn(G)\leq \De(G)+1$. The following theorem due to M\'a\v{c}ajov\'a, Raspaud and \v{S}koviera \cite{MaRS2015} generalizes the famous theorem of Brooks \cite{Brooks41}.

\begin{theorem}
\label{Theorem:Brooks}
Let $G$ be a signed graph, whose underlying graph is simple and connected. If $G$ is not a balanced complete graph, a balanced odd cycle, or a unbalanced even cycle, then $\scn(G)\leq \De(G)$.
\end{theorem}

The aim of this paper is to extend this theorem to arbitrary signed graphs and to prove a list version as well as a degree version of this fundamental result. To accomplish this, we have to take into account two more types of forbidden signed graphs.

A signed graph $G$ is called a {\em brick} if $G$ is a balanced complete graph, a balanced odd cycle, an unbalanced even cycle, a $2K_n$ for an integer $n\geq 2$, or a $2C_n$ for an odd integer $n\geq 3$.

\bigskip

\centerline{\bf Signed list chromatic number}

\medskip

Let $G$ be a signed graph, let
$f:V(G) \to \nato$ be a function, and let
$k\geq 0$ be an integer. A {\em list-assignment}
$L$ of $G$ is a mapping that assigns to every vertex $v$ of $G$ a
set (list) $L(v)$ of colors, i.e., $L(v)\subseteq \mz$.
We say that $L$ is an {\em $f$-assignment}
if $|L(v)|=f(v)$ for all $v\in V$, and a
{\em $k$-assignment} if $|L(v)|=k$ for all $v\in V$, respectively. An {\em $L$-coloring}
of $G$ is a coloring $\f$ of $G$
such that $\f(v)\in L(v)$ for
all $v\in V$. If $G$ admits an $L$-coloring, then $G$ is said to be
{\em $L$-colorable} or {\em list-colorable} if it is clear to
which list-assignment we refer.
A signed graph $G$ is defined to be
{\em $f$-list-colorable} if $G$ is $L$-colorable
for every $f$-assignment $L$ of $G$. When $f(v)=k$ for all $v\in
V$, the corresponding term becomes {\em $k$-list-colorable} or {\em $k$-choosable}.
The {\em signed list-chromatic number} or {\em signed choice number} of $G$, denoted by
$\sln(G)$, is the least number $k\geq 0$ for which $G$ is
$k$-list-colorable.

It is notable that the signed list-chromatic number of a positive signed graph coincides with the list-chromatic number of its underlying graph. Furthermore, it is easy to show that switching equivalent signed graphs have the same list-chromatic number.

A coloring of a signed graph with color set $\Ga_k$ may be regarded as a
list-coloring for the constant list assignment $L$ with $L(v)=\Ga_k$ for all $v\in V(G)$. Thus every signed graph $G$ satisfies $\scn(G)\leq \sln(G)$. It follows from results by Erd\H{o}s, Rubin and Taylor \cite{ERT79} that there are positive signed bipartite graphs whose signed list-chromatic number is arbitrarily large. Based on the usual sequential coloring argument, it is easy to show that the following result holds.

\begin{proposition}
\label{Proposition:Basic}
Every signed graph $G$ is $f$-list-colorable, where $f(v)=d_G(v)+1$ for all $v\in V(G)$. Furthermore, every signed graph $G$ satisfies
$$\scn(G)\leq \sln(G)\leq \col(G)\leq \De(G)+1.$$
\end{proposition}

A signed graph is called {\em degree choosable} if $G$ is $f$-list-colorable
for the degree function $f=d_G$, that is, $f(v)=d_G(v)$ for all $v\in V(G)$. The next theorem, which characterizes degree choosable signed graphs, is one of the main results of this paper. The theorem is an extension of a similar characterization of degree choosable unsigned graphs due to Erd\H{o}s, Rubin and Taylor \cite{ERT79}.

\begin{theorem}
\label{Theorem:DegreeChoosable}
Let $G$ be a connected signed graph. Then $G$ is not degree choosable if and only if each block of $G$ is a brick.
\end{theorem}

Let $G$ be a (signed) graph. Recall that a {\em block} of $G$ is a maximal connected subgraph of $G$ that has no separating vertex. We denote by $\cB(G)$ the set of all blocks of $G$. If $G$ is a connected graph with no separating vertex, then $\cB(G)=\{G\}$. If $G$ is the union of two (signed) graphs $G_1$ and $G_2$ having only one vertex in common, then $\cB(G)=\cB(G_1) \cup \cB(G_2)$, provided that $|G_i|\geq 2$ for $i=1,2$. Any two distinct blocks of $G$ have at most one vertex in common; and a vertex of $G$ is a separating vertex of $G$ if and only if it belongs to more than one block of $G$. A block of $G$ which contains at most one separating vertex of $G$ is called an {\em end-block} of $G$. It is well known that if $G$ contains a separating vertex, then $G$ has at least two end-blocks.

\section{Characterizing uncolorable pairs}

In this section, we shall prove Theorem~\ref{Theorem:DegreeChoosable}. For this reason, we shall characterize the structure of so-called uncolorable pairs.

We call $(G,L)$ an {\em uncolorable pair} if $G$ is a signed connected  graph, $L$ is a list-assignment of $G$ satisfying $|L(v)|\geq d_G(v)$ for all $v\in V(G)$, and $G$ is not $L$-colorable. To characterize such uncolorable pairs, we shall use the following reduction. Let $(G,L)$ be an uncolorable pair. For a vertex $v$ of $G$, we denote by $N_G^+(v)$ the set of all vertices $w$ of $G$ such that $E_G(v,w)$ contains a positive edge. Similarly, we denote by $N_G^-(v)$
the set of all vertices $w$ of $G$ such that $E_G(v,w)$ contains a negative edge. Now, suppose that $|G|\geq 2$. Let $v$ be a non-separating vertex of $G$ and let $c\in L(v)$ be a color. For the signed graph $G'=G-v$ let $L'$ be the list-assignment with
$$L'(u)=
\left\{ \begin{array}{ll}
L(v) \sm \{c,-c\} & \mbox{\rm if } u\in N_G^+(v) \cap N_G^-(v),\\
L(v) \sm \{c\} & \mbox{\rm if } u\in N_G^+(v) \sm N_G^-(v)\\
L(v) \sm \{-c\} & \mbox{\rm if } u\in N_G^-(v) \sm N_G^+(v)\\
L(v) & \mbox{\rm otherwise.}
\end{array}
\right.$$
Then it is easy to check that $|L'(u)|\geq d_{G'}(u)$ for all $u\in V(G')$ and $G'$ is not $L'$-colorable. So $(G',L')$ is an uncolorable pair and we write $(G',L')=(G,L)/(v,c)$.

\begin{theorem}
\label{Theorem:unfarbbaresPaar}
Let $(G,L)$ be an uncolorable pair. Then the following statements hold:
\begin{itemize}
\item[{\rm (a)}] $|L(v)|=d_G(v)$ for all $v\in V(G)$.
\item[{\rm (b)}] Any two vertices are joined by exactly one edge or by a pair of differently signed parallel edges.
\item[{\rm (c)}] If $G$ is a block, then every edge $e\in E_G(v,w)$ satisfies $L(v)=\sg(e)L(w)$. Furthermore, the following statements holds:
\begin{itemize}
\item[{\rm (c1)}] If $\sg=\pl$, then there exists a color set $C$ such that $L(v)=C$ for all $v\in V(G)$.
\item[{\rm (c2)}] If $\sg\not= \pl$, then there exists a symmetric color set $C$ such that $L(v)=C$ for all $v\in V(G)$, or $G$ is balanced with parts $X, Y$ and there exists a color set $C$ such that $L(v)=C$ for all $v\in X$ and $L(v)=-C$ for all $v\in Y$.
\end{itemize}
As a consequence, $G$ is regular.
\item[{\rm (d)}] If $G$ is a block with $\mu(G)\geq 2$, then $G$ is $r$-regular for an even number $r\geq 2$.
\item[{\rm (e)}] Every block of $G$ is a brick.
\end{itemize}
\end{theorem}
\begin{proof}
The proof is by induction on the order $|G|$ of $G$. It is obvious that the statements are true if $|G|=1$, so we can assume that $|G| \geq 2$.

In order to prove (a), we choose an arbitrary vertex $v \in V(G)$. The signed graph $G$ is connected and contains at least two vertices, so $G$ contains a non-separating vertex $w\not=v$ and $L(w) \neq \ems$. For any color $c \in L(w)$ we obtain an uncolorable pair $(G',L')=(G,L)/(w,c)$ by using the previously mentioned reduction. The induction hypothesis then implies $|L'(v)|=d_{G'}(v)$ and, therefore, $|L(v)| = d_G(v)$.

Suppose, contrary to our claim, that statement (b) is false. By reason of symmetry, we may assume that there are vertices $v,w \in V(G)$ which are joined by at least two positive edges. Deleting one of the two edges, results in a signed graph $G'$ such that $(G',L)$ is an uncolorable pair with $|L(v)|>d_{G'}(v)$, giving a contradiction to (a).

For the proof of (c), assume that $G$ is a block. First, let $e\in E_G(v,w)$ be an arbitrary edge of $G$. To see that  $L(v) = \sg(e) L(w)$, suppose this is false. By symmetry, we may assume that there is a color $c \in L(v) \sm \big( \sg(e) L(w) \big)$. Nevertheless, for the uncolorable pair $(G',L')=(G,L)/(v,c)$ we obtain that $d_{G'}(w) < |L'(w)|$ which yields a contradiction to (a).

Next, assume that $\sg=\pl$, that is, $G$ is a positive signed graph. Then for every edge $e\in E_G(v,w)$ we obtain that $L(v)=L(w)$. Since $G$ is connected, this implies that there exists a color set $C$ such that $L(v)=C$ for all $v\in V(G)$. Thus, (c1) holds.

Now, assume that $\sg \neq \pl$. We distinguish two cases. First, suppose that there exists a color set $C$ such that $L(v)=C$ for all $v\in V(G)$. By assumption, there is a negative edge $e$ in $G$, say $e\in E_G(v,w)$. Then $L(v)=-L(w)$ implying that the color set $C$ is symmetric. Second, suppose that $L$ is not a constant list-assignment. Since $G$ is connected and $L(v)=\pm L(w)$ depending on whether $E_G(v,w)$ contains a positive or negative edge, we conclude that there is a color set $C$ such that $L(v)= \pm C$ for all $v \in V(G)$. The color set $C$ is not symmetric, because
$L$ is not the constant list-assignment. Let $X=\set{v\in V(G)}{L(v)=C}$ and $Y=\set{v\in V(G)}{L(v)=-C}$. Since $G$ contains a negative edge and the color set $C$ is not symmetric, the sets
$X$ and $Y$ are non-empty and form a partition of $V(G)$. Due to the fact that any edge $e\in E_G(v,w)$ satisfies $L(v)=\sg(e)L(w)$, we conclude that an edge $e$ of $G$ is negative if and only if $e$ belongs to $E_G(X,Y)$. So $G$ is a balanced graph with parts $X$ and $Y$. This proves (c2).

As a consequence, we obtain that $|L(v)|$ is the same for all vertices $v$ of $G$. By (a), this implies that $G$ is regular. Thus, the proof of statement (c) is complete.

For the proof of (d) assume that $G$ is a block and $\mu(G)\geq 2$. By (c) it follows that $G$ is $r$-regular for an integer $r\geq 0$. Suppose, to the contrary, that $r$ is odd. Since $\mu(G)\geq 2$, it follows from (b) that $G$ contains a pair of differently signed parallel edges. Consequently, $\sg \not= \pl$ and $G$ is unbalanced. Hence, (c2) implies that there exists a symmetric color set $C$ such that $L(v)=C$ for all $v\in V(G)$. By (a), $|C|=r$ and, since $C$ is symmetric and $r$ is odd, $0\in C$. Now let $e,e'\in E_G(v,w)$ be two differently signed parallel edges. Then $(G',L')=(G,L)/(v,0)$ is an uncolorable pair with $|L'(w)|> d_{G'}(w)$, giving a contradiction to (a).
This proves (d).

It remains to prove statement (e). In order to analyze the block structure of the uncolorable pair $(G,L)$ with $|G|\geq 2$, we distinguish two cases.

\case{1}{$G$ contains a separating vertex $v$.} Then $G$ contains at least two end-blocks, say  $B_1$ and $B_2$. Each end-block $B_i$  contains a non-separating vertex $v_i$ of $G$. Obviously, there is a color $c_i\in L(v_i)$ and the induction hypothesis applied to the uncolorable pair  $(G_i,L_i)=(G,L)/(v_i,c_i)$ yields that every block of $G_i$ is a brick.
Since every block $B \neq B_i$ of $G$ is also a block of $G-v_i$, we conclude that every block of $G$ is a brick.

\case{2}{$G$ contains no separating vertex.} Then $G$ is a block and, by (c), $G$ is $r$-regular for an integer $r\geq 1$. Furthermore, it follows from (a) and (c) that there exists a color set $C$ of cardinality $r$ such that
$L(v)=\pm C$ for all $v \in V(G)$ (including the case that $L(v)=C$ for all $v \in V(G)$).

Let $v$ be an arbitrary vertex of $G$ and choose a color $c \in L(v)$.  By applying the induction hypothesis on $(G',L')=(G,L)/(v,c)$, we conclude that every block of $G'$ is a brick. This leads to the following two subcases.

\case{2.1}{$G-v$ contains no separating vertex.} Then $G'=G-v$ is a block and, by (b) and (c), $G'$ is regular of degree $r-2$ or $r-1$.

First, assume that $G'$ is $(r-2)$-regular. Since $G$ is $r$-regular, this implies that in $G$ every vertex of $G'$ is joined to $v$ by two parallel edges. Thus,
$$r = 2 \cdot |V(G)| - 2,$$
which is equivalent to
$$r-2 = 2 \cdot (|V(G)| - 1) - 2=2 \cdot |V'(G)| - 2.$$
According to (b), this is solely possible if $G = 2K_n$ and $G'=2K_{n-1}$.

Now, assume that $G'$ is $(r-1)$-regular. Then it follows from (d) that $G$ is simple (either $r$ is odd and both $G$ and $G'$ are simple or $r-1$ is odd). Hence, $\ul{G}$ as well as $\ul{G}'$ are complete graphs. It remains to verify that $G$ is balanced. This is evident if $G$ is positive or $|G|=2$. So we may suppose that $\sg\not=\pl$ and $|G|\geq 3$. By (c2), it suffices to consider the case that there exists a symmetric color set $C$ such that $L(v)=C$ for all $v\in V(G)$. As $\ul{G}$ is complete, it follows from (a) that $|C|=|G|-1\geq 2$. Consequently, $C$ contains a positive color $c$. Since $\sg\not= \pl$, there is a negative edge $e$ in $G$, say $e\in E_G(v,w)$. Consequently, $(G',L')=(G,L)/(v,c)$ is an uncolorable pair, where $\ul{G}'$ is a complete graph and $L(w)=C\sm \{-c\}$. Since the color set $C'=C\sm \{-c\}$ is not symmetric, we conclude from statement (c) applied to $(G',L')$ that $G'$ is a balanced graph with parts $X,Y$ such that $L(u)=C'$ for all $u\in X$ and $L(u)=-C'$ for all $u\in Y$. By the construction of $L'$, it follows that every edge of $E_G(v;X)$ is negative and every edge of $E_G(v,Y)$ is positive. Hence, $G$ is a balanced graph with parts $X, Y \cup \{v\}$.

\case{2.2}{$G-v$ contains a separating vertex.} Consequently, $G'$ contains at least two end-blocks and every end-block is regular of degree $r-1$ or $r-2$.

First, assume that $G'$ contains two end-blocks, say $B_1$ and $B_2$, which are both $(r-1)$-regular. Then we can easily conclude from (d) that $B_1$ and $B_2$ are simple (either $r$ is odd and hence $G$ is simple, or $r-1$ is odd and by repeated application of our reduction, we obtain an uncolorable pair $(B_i,L_i)$ implying that $B_i$ is simple). Consequently, at least $r-1$ vertices of $B_i$ are adjacent to $v$ in $G$. This leads to $2(r-1) \leq d_G(v)=r$ and, therefore, $r=2$. This shows that $G$ is a cycle.

It remains to show that the cycle $G$ is a brick, that is, $G$ is a balanced odd cycle or an unbalanced even cycle. If $\sg=\pl$, then it follows from (a) and (c1) that there is a set $C$ of two colors such that $L(v)=C$ for all $v\in V(G)$. Since the cycle $G$ is not $L$-colorable, we conclude that $G$ is an odd cycle and hence a brick. If $\sg\not=\pl$, then it follows from (a) and (c2) that $G$ is balanced with parts $X, Y$ and there is a set $C$ of two colors such that $L(v)=C$ for all $v\in X$ and $L(v)=-C$ for all $v\in Y$, or there is a symmetric set $C$ of two colors such that $L(v)=C$ for all $v\in V(G)$. In the first case, let $G'=G/Y$ and let $L'$ be the list assignment with $L'(v)=C$ for all $v\in V(G)$. Then the cycle $G'$ is not $L'$-colorable, since otherwise an $L'$-coloring $\f'$ of $G'$ would lead to an $L$-coloring $\f=\f'/Y$ of $G$, a contradiction. In this case, we conclude that $G$ is an odd balanced cycle and so $G$ is a brick.  In the second case, we distinguish two subcases. If the sign product of the cycle $G$ is positive, then $G$ is balanced and, as in the former case, we conclude that $G$ is an odd balanced cycle and hence a brick. If the sign product of the cycle $G$ is negative, then we conclude that $|G|$ is even and, therefore, $G$ is an unbalanced even cycle. Otherwise, $G$ would be an odd cycle and so $G$ would be an antibalanced odd cycle implying that $\scn(G)\leq 2$ (by Theorem\ref{Theorem:Harary:antibalanced}), contradicting the fact that $G$ is not $L$-colorable, where $L(v)=\{c,-c\}$ for all $v\in V(G)$ and $c\in \mz$ is a positive color. Consequently, in all cases, we conclude that $G$ is a brick.

Next, assume that $G'$ contains two end-blocks, say $B_1$ and $B_2$, such that $B_1$ is $(r-1)$-regular and $B_2$ is $(r-2)$-regular. Then we conclude, similar as in the previous case, that $B_1$ is simple and so at least $r-1$ vertices of $B_1$ are adjacent to $v$ in $G$. Since at least one vertex of $B_2$ is joined to $v$ by two parallel edges, we obtain $d_G(v)\geq r+1$, which is impossible.

Finally, assume that each end-block of $G'=G-v$ is $(r-2)$-regular. Let $B$ be an arbitrary end-block of $G'$ and let $v'$ be the only separating vertex of $G'$ contained in $B$. As $B$ is $(r-2)$-regular, it follows from (b) that $|B|\geq \frac{r-2}{2}+1$. Furthermore, every vertex $u$ of $B-v'$ is in $G$ joined to $v$ by two parallel edges. Hence, if $k$ is the number of end-blocks of $G'$, we obtain that $d_G(v)\geq 2k \frac{r-2}{2}=k(r-2)$. Since $r\geq 4$ and $d_G(v)=r$, this leads to $k=2$ and $r=4$. Since $v$ was an arbitrarily chosen vertex of $G$, we easily conclude that each block of $G'$ is a $2K_2$ implying that $G=2C_n$ for an integer $n\geq 3$.
Then $G$ is unbalanced and it follows from (a) and (c2) that there is a symmetric color set $C$ such that $|C|=4$ and $L(v)=C$ for all $v\in V(G)$.
As $(G,L)$ is an uncolorable pair, we easily conclude that $n$ is odd. Thus,  the proof of the theorem is complete.
\end{proof}

The above theorem provides a characterization of signed graphs which may occur in uncolorable pairs. Our next aim is to characterize the lists assignments which may occur in uncolorable pairs. The next lemma describes the lists in uncolorable pairs where the signed graph is a brick. The proof of this result is left to the reader, one may use the proof of the above theorem.

\begin{lemma}
\label{Lemma:ColorListsBricks}
Let $B$ be a brick and let $L$ be a list-assignment for $B$. Then the following statements hold:

\begin{itemize}
\item[\rm (a)] If $B$ is a balanced complete graph or a balanced odd cycle with parts $X, Y$, then $(B,L)$ is an uncolorable pair if and only if there is a color set $C$ such that $|C|=\De(B)$, $L(v)=C$ for all $v\in X$, and $L(v)=-C$ for all $v\in Y$.
\item[\rm (b)] If $B$ is an unbalanced even cycle, a $2K_n$ with $n\geq 2$, or a $2C_n$ with $n\geq 3$ odd, then $(B,L)$ is an uncolorable pair if and only if there exists a symmetric color set $C$ such that $|C|=\De(B)$ and $L(v) = C$ for all $v \in V(B)$.
\end{itemize}
\end{lemma}

The following lemma states that we can create a new uncolorable pair from two given uncolorable pairs by merging two of their vertices if we do it properly.

\begin{lemma}
\label{Lemma:UnionList}
Let $(G_1,L_1)$ and $(G_2,L_2)$ be two uncolorable pairs. Suppose that $G_1$ and $G_2$ have only vertex $v$ in common, $G=G_1 \cup G_2$, and that $L$ is the list-assignment of $G$ satisfying
$$L(u)=
\left\{ \begin{array}{ll}
L_i(u) & \mbox{\rm if } u\in V(G_i)\sm\{v\} \mbox{ with } i=1,2,\\
L_1(v) \cup L_2(v) & \mbox{\rm if } u=v.
\end{array}
\right.$$
for all $u\in V(G)$.
If $L_1(v) \cap L_2(v)=\ems$, then $(G,L)$ is an uncolorable pair.
\end{lemma}
\begin{proof}
Clearly, $G$ is connected and $|L(w)|\geq d_G(w)$ for all $w\in V(G)$. It remains to show that $G$ is not $L$-colorable. Suppose this is false and there is an $L$-coloring $\f$ of $G$. Then $\f(v)\in L_i(v)$ for some $i\in \{1,2\}$. But then the restriction $\f_i$ of $\f$ to $G_i$ would be an $L_i$-coloring of $G_i$, a contradiction. Consequently, $G$ is not $L$-colorable and so $(G,L)$ is an uncolorable pair.
\end{proof}

It seems obvious that this method not only applies to one direction. Indeed, it is true that we can create new list-assignments $L_i$ from a preexisting list-assignment $L$ of an uncolorable pair $(G,L)$ such that each block $B_i \in \cB(G)$ along with $L_i$ forms a new uncolorable pair. This circumstances are proven in the next lemma. For a signed graph $G$ and a vertex $v$ of $G$, let $\cB_v(G)=\set{B\in \cB(G)}{v\in V(B)}$.

\begin{lemma}
\label{Lemma:BlockLists}
Let $(G,L)$ be an uncolorable pair. Then for each block $B \in \cB(G)$ there is a list-assignment $L_B$ such that $(B,L_B)$ is an uncolorable pair and $L(v) = \bigcup_{B \in \cB_v(G)}L_B(v)$ for all $v\in V(G)$.
\end{lemma}
\begin{proof}
The proof is by induction on the number of blocks of $G$. If $G$ has only one block, there is nothing to prove. So assume that $G$ has at least two blocks. Then $G$ has an end-block $G_1$ and there is exactly one separating vertex $v$ of $G$ contained in $G_1$. By defining $G_2=G-(V(G_1)\sm \{v\})$, we obtain
that $G=G_1 \cup G_2$, $G_1 \cap G_2 = (\{v\}, \ems)$ and $\cB(G_2)=\cB(G) \sm \{G_1\}$. For $i=1,2$, let $C_i$ be the set of all colors $c \in L(v)$ for which there exists no $L$-coloring $\f_i$ of $G_i$ with $\f_i(v)=c$. If there were a color $c \in L(v) \sm (C_1 \cup C_2)$ this would clearly lead to an $L$-coloring $\f_i$ of $G_i$ with $\f_i(v)=c$ for $i=1,2$ and so $\f = \f_1 \cup \f_2$ would be a proper $L$-coloring of $G$, giving a contradiction. Hence, $L(v)=C_1 \cup C_2$ holds. For the signed subgraph $G_i$ of $G$ with $i=1,2$, we define $L_i$ by
\begin{equation*}
L_i(u)=
\begin{cases}
L(u) & \text{if } u \in V(G_i) \sm \{ v \},\\
C_i & \text{if } u=v.
\end{cases}
\end{equation*}
Due to the choice of $C_i$, the graph $G_i$ is not $L_i$-colorable. Since $(G,L)$ is an uncolorable pair, $|L_i(u)| = |L(u)| \geq d_G(u) = d_{G_i}(u)$ for all $u \in V(G_i) \sm \{ v \}$. The signed graph $G_i$ is not $L_i$-colorable, thus, Theorem~\ref{Theorem:unfarbbaresPaar} yields $|C_i| = |L_i(v)| \leq d_{G_i}(v)$. Because $L(v)$ is the union of $C_1$ and $C_2$, we can additionally conclude that $|C_1| + |C_2| \geq |C_1 \cup C_2| = |L(v)| \geq d_G(v) = d_{G_1}(v) + d_{G_2}(v)$, and hence $|L_i(v)|=|C_i| = d_{G_i}(v)$. This implies that $(G_i,L_i)$ is an uncolorable pair for $i=1,2$. By the induction hypothesis, for each block $B\in \cB(G_2)$ there is a list-assignment $L_B$ of $B$ such that $(B,L_B)$ is an uncolorable pair and $L_2(u)=\bigcup_{B\in \cB_u(G_1)}L_B(u)$ for all $u\in V(G_2)$. Since $\cB(G)=\cB(G_2) \cup \{G_1\}$ and $(G_1,L_1)$ is an uncolorable pair with $L_1(v)\cap L_2(v)=\ems$, the desired result for the blocks of $G$ follows.
\end{proof}

Combining the results of this section, we obtain Theorem~\ref{Theorem:DegreeChoosable} saying that a signed graph $G$ is not degree choosable if and only if each block of $G$ is a brick.
\\

\pff{Theorem~\ref{Theorem:DegreeChoosable}}
{Let $G$ be a connected signed graph. First suppose that every block of $G$ is a brick. Then it follows by induction on the number of blocks of $G$, using Lemma~\ref{Lemma:ColorListsBricks} and Lemma~\ref{Lemma:BlockLists}, that there is a list-assignment $L$ of $G$ such that $(G,L)$ is an uncolorable pair and so $G$ is not degree choosable. Now suppose that $G$ is not degree choosable. Then there is a list-assignment $L$ of $G$ such that $|L(v)|=d_G(v)$ for all $v\in V(G)$ and $G$ is not $L$-colorable. Then $(G,L)$ is an uncolorable pair and it follows from Theorem~\ref{Theorem:unfarbbaresPaar} that every block of $G$ is a brick.
}

\begin{corollary}
\label{Corollary:Brooks2}
Let $G$ be a signed connected graph. If $G$ is not a brick, then $\sln(G)\leq \De(G)$.
\end{corollary}
\begin{proof}
Suppose this is false. Then there is a $\De(G)$-assignment $L$ of $G$ such that $G$ is not $L$-colorable. Consequently, $(G,L)$ is an uncolorable pair. From Theorem~\ref{Theorem:unfarbbaresPaar} it follows that $G$ is regular of degree $\De(G)$ and each block of $G$ is a brick. Since every brick is regular, we then conclude that $G$ has only one block implying that
$G$ is a brick.
\end{proof}

\section{List critical signed graphs}

Let $G$ be a signed graph and let $L$ be a list assignment of $G$. If $G$ is $L$-colorable, then every signed subgraph of $G$ is $L$-colorable, too. The signed graph $G$ is said to be {\em $L$-critical} if $G$ is not $L$-colorable, but every signed proper subgraph of $G$ is $L$-colorable. Furthermore, we say that $G$ is $k$-list-critical if there is a $(k-1)$-list-assignment $L$ for which $G$ is $L$-critical.

Let $k\geq 1$ be an integer. A signed graph $G$ is called {\em $k$-critical}  if $\scn(G)=k$ and every signed proper subgraph $G'$ of $G$ satisfies $\scn(G')\leq k-1$. Clearly, every $k$-critical signed graph $G$ is $L$-critical with respect to the constant list-assignment $L$ satisfying $L(v)=\Ga_{k-1}$ for all $v\in V(G)$. So every $k$-critical signed graph is
$k$-list-critical. If $G$ is a brick, then $G$ is regular, say of degree $r$, and it follows from Theorem~\ref{Theorem:unfarbbaresPaar} that $G$ is $(r+1)$-critical. A signed graph $G$ is called {\em $k$-choice-critical} if $\sln(G)=k$ and every signed proper
subgraph $G'$ of $G$ satisfies $\sln(G')\leq k-1$. Clearly, every $k$-choice-critical signed graph is $k$-list-critical.
The balanced complete graph of order $k$ is an example of a $k$-critical signed graph and for $k=1$ and $k=2$ there are no other $k$-critical signed graphs. Clearly, every signed graph with $\scn=k$ contains a $k$-critical signed subgraph. As a consequence of Theorem~\ref{Theorem:Harary:antibalanced} we obtain the following result.

\begin{lemma}
\label{Lemma:3Critical}
A signed graph $G$ is 3-critical if and only if $G$ is a balanced odd cycle or an unbalanced even cycle.
\end{lemma}
\begin{proof}
If $G$ is a balanced odd cycle or an unbalanced even cycle, then Theorem~\ref{Theorem:unfarbbaresPaar} implies that $G$ is 3-critical. Now assume that $G$ is a 3-critical signed graph. If $G$ contains no balanced odd cycle and no unbalanced even cycle, then $G$ is antibalanced and, by Theorem~\ref{Theorem:Harary:antibalanced}, $\scn(G)\leq 2$, a contradiction. So $G$ contains a signed subgraph $H$ which is a balanced odd cycle or an unbalanced even cycle. Then $\scn(H)=3$ and, since $G$ is 3-critical, we conclude that $G=H$.
\end{proof}

Before we establish some basic facts about list-critical signed graphs, we need some further notations. For a set $X\subseteq V(G)$, let $G[X]$ denote the signed subgraph of $G$ induced by $X$ so that $V(G[X])=X$ and $E(G[X])=E_G[X]$. Furthermore, for a vertex $v\in v(G)$ let $d_G(v:X)=|E_G(v,X)|$.

\begin{lemma}
\label{Lemma:ListCritical}
Let $G$ be an $L$-critical signed graph for an list-assignment $L$ of $G$. Furthermore, let $H=\set{v\in V(G)}{d_G(v)>|L(v)|}$, let $F=V(G)\sm H$, and let $\ems\not= X\subseteq F$. Then the following statements hold:
\begin{itemize}
\item[{\rm (a)}] $d_G(v)=|L(v)|$ for all $v\in X$.
\item[{\rm (b)}] Every block of $G[X]$ is a brick.
\item[{\rm (c)}] If $L$ is a $(k-1)$-list-assignment with $k\geq 1$, then $H\not=\ems$ or $G$ is a brick. Furthermore, if $G[X]$ contains a $K_k$, then $G$ is a balanced complete graph of order $k$.
\end{itemize}
\end{lemma}
\begin{proof}
For the proof it suffices to consider the case when $G[X]$ is connected. Clearly, $Y=V(G) \sm X$ is a proper subset of $V(G)$. Since $G$ is $L$-critical, this implies that there is an $L$-coloring $\f$ of $G[Y]$. For the remaining signed graph $G'=G[X]$, let $L'$ be the list-assignment defined by
$$L'(v)=L(v) \sm (\set{\f(u)}{u\in Y \cap N_G^+(v)} \cup \set{-\f(u)}{u\in Y \cap N_G^-(v)})$$
for all $v\in V(G')$. Since $G$ is not $L$-colorable, it follows that $G'$ is not $L'$-colorable. Since $X\subseteq F$, we conclude that every vertex $v\in X$ satisfies
$$|L(v)|\geq d_G(v)=d_{G'}(v)+d_G(v:Y)$$
and, therefore,
$$|L'(v)|\geq |L(v)|-d_G(v:Y)\geq d_{G'}(v).$$
Consequently, $(G',L')$ is an uncolorable pair and it follows from Theorem~\ref{Theorem:unfarbbaresPaar} that every block of $G'$ is a brick and $|L'(v)|=d_{G'}(v)$ for all $v\in X$, which implies that
$$|L(v)|=d_{G'}(v)+d_G(v:Y)=d_G(v)$$
for all $v\in X$. This proves (a) and (b). For the proof of (c) assume that $|L(v)|=k-1$ for all $v\in V(G)$. If $H\not=\ems$, then there is nothing to prove. So assume that $H=\ems$. Then $G=G[F]$ and, since $G$ is $L$-critical, we conclude that $G$ is connected. From (a) and (b) it follows that every block of $G$ is a brick and $G$ is regular of degree $k-1$. Since each brick is regular, this implies that $G$ consists only of one block and, therefore, $G$ is a brick. If $G[X]$ contains a complete graph of order $k$, then $d_G(v)=k-1$ for every vertex $v$ belonging to this complete graph. This obviously implies that this $K_k$ is a component of $G$. Since the $L$-critical graph $G$ is connected, this implies that $G$ is a complete graph of order $k$ and $(G,L)$ is an uncolorable pair. By Theorem~\ref{Theorem:unfarbbaresPaar}, it follows that $G$ is a balanced complete graph of order $k$.
\end{proof}

Let $G$ be a $k$-list-critical signed graph. Then Lemma~\ref{Lemma:ListCritical} implies that $\de(G)\geq k-1$ and so $|E(G)|\geq \tfrac{1}{2}(k-1)|V(G)|$, where equality holds if and only if $G$ is a brick.  Furthermore, it follows from Lemma~\ref{Lemma:ListCritical} that if $G'$ is the signed subgraph of a $k$-list-critical signed graph $G$ induced by the vertices having degree $k-1$ in $G$, then $G'=\ems$ or every block of $G$ is a brick. This leads to an improvement of the trivial lower bound for the number of edges in a $k$-list-critical graph with $k\geq 4$. Theorem~\ref{Theorem:EdgeNumber} provides such an improvement for the class of simple signed graphs; this theorem is a counterpart to a result about the number of edges in color critical graphs obtained by Gallai \cite{Gallai63a}.

For an integer $k\geq 4$, let $\cT_k$ denote the class of signed connected graphs $T$ such that $\mu(T)\leq 1$, $\De(T)\leq k-1$, every block of $T$ is a brick, and $T$ is not a balanced complete graph of order $k$. The following result is an extension of a similar result due to Gallai \cite{Gallai63a}.

\begin{lemma}
\label{Lemma:Gallaitree}
Let $k\geq 4$ be an integer. Then every signed graph $T\in \cT_k$ satisfies $$(k-2+\frac{2}{k-1})|V(T)|-2|E(T)|\geq 2$$
\end{lemma}
\begin{proof}
Throughout the proof let
$$r=k-2+\frac{2}{k-1}$$
and, for $T\in \cT_k$, let
$$m(T)=r|V(T)|-2|E(T)|=\sum_{v\in V(T)}(r-d_T(v)).$$
Our aim is to show that if $T\in \cT_k$, then $m(T)\geq 2$. The proof is by induction on the number of blocks of $T$. If $T$ consists of one block $B$, then $B$ is a balanced complete graph of order $b$ with $1\leq b \leq k-1$, or $B$ is a balanced odd cycle of order at least five, or $B$ is an unbalanced even cycle of order at least four.
If $B$ is a balanced complete graph of order $b$ with $1\leq b \leq k-1$, then we conclude that
$$m(B)=b(r-b+1)
\left\{ \begin{array}{ll}
\geq r & \mbox{\rm if } 1\leq b \leq k-2,\\
=2 & \mbox{\rm if } b=k-1.
\end{array}
\right.$$
Otherwise, $\ul{B}$ is a cycle of order at least four, and we conclude that
$m(B)\geq (r-2)4\geq r\geq 2$, since $k\geq 4$ and so $r\geq 2$.

For the induction step suppose that $T$ has at least two blocks. Let $\cB$ denote the set of all end-blocks of $T$. First suppose that there is a block $B\in \cB$ such that $\ul{B}$ is distinct from $K_{k-1}$. Let $v$ denote the only separating vertex of $T$ contained in $B$, and let $T'=T-(V(B)\sm \{v\})$. Then $T'\in \cT_k$ and the induction hypothesis implies that $m(T')\geq 2$. Since $T'$ and $B$ have only vertex $v$ in common, we obtain that $m(T)=m(T')+m(B)-r$. Clearly, $B$ belongs to $\cT_k$ and $\ul{B}\not=K_{k-1}$ implying that $m(B)\geq r$. Summarizing, this yields $m(T)\geq 2$.

Now suppose that every block $B\in \cB$ satisfies $\ul{B}=K_{k-1}$. Let $B$ be an arbitrary end-block of $T$, and let $v$ be the only separating vertex of $T$ contained in $B$. Since $\De(T)\leq k-1$ and $T$ has at least two end-blocks, we conclude that $v$ is contained in a block $B'$ such that $\ul{B'}=K_2$ and $T'=T-V(B)$ belongs to $\cT_k$. Then $m(T)=m(T')+m(B)-2$ and, by the induction hypothesis, $m(T')\geq 2$. Since $\ul{B}=K_{k-1}$, we obtain that $m(B)=2$. Summarizing, this yields $m(T)\geq 2$.
\end{proof}

\begin{theorem}
\label{Theorem:EdgeNumber}
Let $k\geq 4$ be an integer. If $G$ is a $k$-list-critical signed graph such that $\mu(G)\leq 1$ and $G$ is not a balanced complete graph of order $k$, then
$$2|E(G)|\geq (k-1 +\frac{k-3}{k^2-3})|V(G)|.$$
\end{theorem}
\begin{proof}
Let $V=V(G)$, and for a subset $X$ of $V$ let $e(X)=|E_G(X)|$. Our aim is to show that
$$2e(V)\geq (k-1 +\frac{k-3}{k^2-3})|V|.$$
Let $H=\set{v\in V(G)}{d_G(v)\geq k}$ and $F=\set{v\in V(G)}{d_G(v)=k-1}$. If $F=\ems$, then $2e(V)=k|V|$ and we are done. So assume that $F\not=\ems$. Then it follows from Lemma~\ref{Lemma:ListCritical} that $V(G)=H \cup F$ and every block of $G[F]$ is a brick. Furthermore, since $G$ is not a balanced complete graph of order $k$, we conclude from  Lemma~\ref{Lemma:ListCritical}(c) that $G[F]$ contains no $K_{k}$ as a subgraph. Since $\De(G[F])\leq k-1$ and $\mu(G)\leq 1$, this implies that each component of $G[F]$ belongs to $\cT_k$. From Lemma~\ref{Lemma:Gallaitree} it then follows that
$$(k-2+\frac{2}{k-1})|F|-2e(F)\geq 2.$$
On the one hand, this implies that
\begin{eqnarray*}
2e(V) &=& 2e(H)+2(k-1)|F|-2e(F)\\
&\geq& 2(k-1)|F|-2e(F) \geq (k-\frac{2}{k-1})|F|.
\end{eqnarray*}
On the other hand, we obtain that
\begin{eqnarray*}
2e(V) &=& (k-1)|V|+|H|+\sum_{v\in H}(d_G(v)-k)\\
&\geq& (k-1)|V|+|H|=k|V|-|F|.
\end{eqnarray*}
Multiplying the second inequality with $k-\frac{2}{k-1}$ and adding the first inequality, yields
$$2e(V)(1+k-\frac{2}{k-1})\geq k(k-\frac{2}{k-1})|V|$$
which is equivalent to
$$2e(V)\left(\frac{k^2-3}{k-1}\right)\geq \left(\frac{(k-1)(k^2-3)+k-3}{k-1}\right)|V|$$
and, hence, to
$$2e(V)\geq (k-1 +\frac{k-3}{k^2-3})|V|.$$
This completes the proof.
\end{proof}

\section{Concluding remarks}

Our results show that coloring of signed graphs behave in a very much similar way as ordinary vertex colorings of graphs. A result concerning the chromatic number whose proof mainly relies on the sequential coloring argument can quite often be transformed into a similar result about the signed chromatic number or the signed choice number of simple signed graphs. Typical examples of this fact are the characterization of degree choosable signed graphs and the characterization of uncolorable pairs. The proof of Theorem~\ref{Theorem:unfarbbaresPaar} resembles the proof of a similar result from \cite{KostochkaSW96}. Also Thomassen \cite{Thomassen94} famous proof that every planar graph is 5-choosable can be applied to signed simple planar graphs as shown by Jing, Kang and Steffen \cite{JingKS2015}. That every signed simple planar graph $G$ satisfies $\scn(G)\leq 5$ was observed by M\'a\v{c}ajov\'a, Raspaud and \v{S}koviera \cite{MaRS2015}. They also conjectured that we have in fact $\scn(G)\leq 4$ for every signed simple  planar graph.

It follows from Proposition~\ref{Proposition:Basic} that the signed choice number of every signed simple graph whose underlying graph can be embedded on a surface of Euler genus $\ep\geq 1$ is at most the Heawood number $H(\ep)=\lfloor(7+\sqrt{24\ep+1})/2\rfloor$. As proved by B\"ohme, Stiebitz and Mohar \cite{BohmeMS}, if $G$ is a graph embedded on a surface of Euler genus $\ep$ with $\ep\geq 1$ and $\ep\not=3$, then the choice number of $G$ is at most $H(\ep)$ and equality holds if and only if $G$ contains a complete subgraph of order $H(\ep)$. That the result also holds for $\ep=3$ was proved by Kr\'al' and \v{S}krekovski \cite{KralS}. It seems very likely that a similar map color theorem holds for the signed choice number, that is, if $G$ is a signed simple graph embedded on a surface of Euler genus $\ep$, where $\ep=2$ or $\ep\geq 4$, then $\sln(G)\leq H(\ep)$ and equality holds if and only if $G$ contains a balanced complete subgraph of order $H(\ep)$. If this is not true one may consider a minimal counterexample, that is, a signed simple graph $G$ embedded on a surface of
Euler genus $\ep$ such that $\sln(G)=k$ with $k=H(\ep)$ and $G$ does not contain a balanced complete subgraph of order $k$. Then $G$ is $k$-list-critical. If $n=|V(G)|$ and $m=|E(G)|$, we obtain from Euler's Formula that $2m\leq 6n - 12 +6\ep$ and, by Theorem~\ref{Theorem:EdgeNumber}, we obtain that
$$2m\geq \left( k-1 +\frac{k-3}{k^2-3} \right) n.$$
If $n\geq k+4$ and $\ep=2$ or $\ep\geq 4$, this leads to a contradiction as in the proof of the corresponding map color theorem in \cite{BohmeMS}. The open question is whether we can handle the case $k+1 \leq n \leq k+3$ similar as in \cite{BohmeMS} or whether we can improve the bound in Theorem~\ref{Theorem:EdgeNumber} to $2|E(G)|\geq (k-1)|V(G)|+k-3$.

If $G$ is an ordinary graph with chromatic number $k$, then in any optimal coloring of $G$ all $k$ colors are used. Furthermore, $\chi(G)\leq k$ if and only if $V(G)$ is the union $k$ pairwise disjoint independent sets (possibly empty). From the proof of Theorem~\ref{Theorem:G=2H} it follows that there are signed graphs $G$ with $\scn(G)=2k-1$ such that only $k+1$ colors are used in an optimal coloring of $G$. Clearly, if $G$ is a signed graph, then $\scn(G)\leq 2k$ if and only if $G$ has a partition into $k$ disjoint antibalanced signed subgraphs, and $\scn(G)\leq 2k+1$ if and only if $G$ has a partition into $k$ disjoint antibalanced signed graphs and one edgeless signed subgraph. So the colors different from zero form pairs and the color zero plays a particular role.

An edge coloring of an ordinary graph can be viewed as a vertex coloring of its line graph. For a signed graph $G$, the {\em signed line graph} of $G$ is the signed simple graph $H=L(G)$ such that $V(H)=E(G)$, $E(H)$ consists of all pairs $ee'$ of distinct edges of $G$ having a common end in $G$ and $\sg_H(ee')=\sg_G(e)\sg_G(e')$. Clearly, if $E^+$ is the set of positive edges of $G$ and $E^-$ is the set of negative edges of $G$, then $H=L(G)$ is a balanced signed graph with parts $E^+$ and $E^-$. This implies that $\scn(H)=\chi(\ul{H})$ and $\ul{H}$ is the ordinary line graph of $\ul{G}$. If $G$ is a simple signed graph, then it follows from Vizing's theorem that $\De(G) \leq \scn(L(G)) \leq \De(G)+1$.

\end{document}